\documentclass[a4paper, 11pt]{article}
\usepackage[latin1]{inputenc}
\usepackage[english]{babel}

\usepackage{mathrsfs}
\usepackage{amsmath}
\usepackage{amsfonts}
\usepackage{amsthm}
\usepackage{latexsym}
\usepackage{amsthm}
\usepackage[dvips]{graphics,color}
\usepackage{graphicx}

\newcommand{\led}{\langle}
\newcommand{\red}{\rangle}
\newcommand{\help}{\mathbb N}
\newcommand{\Ups}{\Upsilon}
\newcommand{\ex}{\mathbb E}

\newtheorem{claim}{Claim}

\title{ \Large\textsc{FIRST-PASSAGE PERCOLATION WITH \\ EXPONENTIAL TIMES ON A 
LADDER\footnote{MSC Primary 82B43. MSC Secondary 60K35.}}}
\author{Henrik Renlund\footnote{
Math.\ Department, Uppsala University, P.0.\ Box 480, 751\,06 Uppsala, Sweden}}

\date{}

\begin{document}

\maketitle

\begin{abstract}
We consider first-passage percolation on a ladder, i.e.\ the graph $\help\times\{0,1\}$
where nodes at distance 1 are joined by an edge, and
the times are exponentially i.i.d.\ with mean 1. We find an appropriate
Markov chain to calculate an
explicit expression for  the \emph{time constant} whose numerical
value is $\,\approx 0.6827$. This time constant is the long-term
average inverse speed of the process. We also calculate the average
residual time. 
\end{abstract}

\section{Introduction}
Consider a graph $G$ with vertex (node) set $V$ and edge set $E\subset V\times V$. An
(undirected) edge $e=\led v,v' \red =\led v',v \red$ joins vertex $v$ and $v'$.
A path $\pi(v,v')$ between $v$ and $v'$, if it exists, is an alternating sequence of vertexes 
and edges $(v_0,e_1,v_1,\ldots,e_n,v_n)$ such that $e_i=\led v_{i-1},v_i \red$ for $i=1,\ldots,n$, 
$v_0=v$ and $v_n=v'$. 

Associate with each edge $e$ a non-negative random variable $\xi_e$. Let the time of
a path $\pi(v,v')$ be 
\[ T\pi(v,v')=\sum_{e\in \pi(v,v')} \xi_e\quad\mbox{and let}\quad
 T(v,v')=\inf_{\pi(v,v')} T\pi(v,v') \]
be the shortest time of any path between $v$ and $v'$. This may be
called the first passage time of $v'$ (with respect to $v$) and
is the subject of investigation in first passage percolation,
see e.g.\ \cite{SW78}. 

We will think of this as a model for a contagious disease and refer 
to the first passage time of a node as the time of infection.
Usually, we think of some node $v$ (or more generally a set of nodes) as infected at time zero.
The times $T(w)=T(v,w)$, $w\in V$, lets us know when nodes $w\in V$ 
are infected. 

One natural question is; how fast is the spread of  the infection? The present paper considers
the equivalent question of the inverse speed of percolation on a ladder, i.e.\ how long
does it take for the infection to spread one more step up the ladder, averaged over time. 
We also calculate the average inverse speed in another sense; namely the average residual time.

\section{First-passage percolation on the ladder}

First-passage percolation on a ladder has previously been studied in \cite{FGS06}, which
gives a method of calculating the time constant when the times associated with 
edges have a discrete distribution, as well as a method for getting arbitrarily 
good bounds for the same quantity
when the distribution is continuous (and well behaved). Our paper gives an exact formula for
the time constant only when the times are exponential and the method we use relies heavily on the
special properties of this distribution.

It has been brought to our attention that during the referee period of the present paper 
\cite{Sch09} has emerged, which also contains the percolation rate
(inversion of the time constant), although it is achieved through other methods. 

Let $L$ be a ladder, by which we mean a graph with vertex set $V=\help\times \{0,1\}$ and 
where the edge set
$E$ consists of $\led v,v'\red$ such that $v$ and $v'$ are at distance one
from each other, see Figure \ref{ladder}. 
By $\help$ we mean
the set of non-negative integers $\{0,1,2,\ldots\}$.

\begin{figure}[htb] \begin{center}
 \includegraphics[scale=0.64]{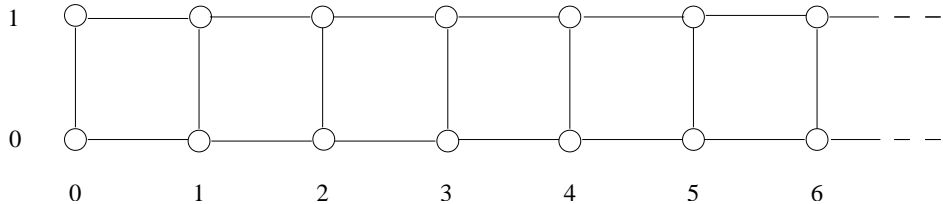}
\end{center}\caption{Part of the ladder $L$.}\label{ladder}\end{figure} 

We will consider the model where the nodes $(0,0)$ and $(0,1)$ are infected
at time zero and each edge is associated with a random variable that
has an exponential distribution with mean 1. Random variables associated
to different edges are independent. 

For a node $v=(x,y)$ we will say
that $v$ is at height $x\in\help$ and level $y\in\{0,1\}$. 
Recall that $T[(x,y)]$ is the infection time of node $(x,y)$. Let
\[ N_t^{(i)}=\sup\{x: T[(x,i)]\leq t\},\quad i=0,1, \]
i.e.\ let $N_t^{(0)}$ and $N_t^{(1)}$ denote the height of the infection
at level $0$ and $1$, respectively, at time $t$. 

We will let $N_t=\max\{N_t^{(0)},N_t^{(1)}\}$ denote the height of the infection, i.e.\
the largest height of any infected node, at time $t$.

We aim to calculate the time constant $\tau$, defined as the limiting inverse percolation rate, 
$1/\tau =\lim_t N_t/t$.
We will not do this directly but calculate the percolation rate
at ``time infinity'', or rather when a process related to $N_t$ has reached stationarity.

More explicitly, we will employ an observation by Alm and 
Janson \cite{snack} that arose in connection with their joint
work on one-dimensional lattices \cite{AJ90}, namely; it might be fruitful to
consider  the \emph{front} $F_t=|N_t^{(0)}-N_t^{(1)}|$
at time $t\geq 0$. Front thus means the absolute value of the difference
in height between the highest infected nodes on each level, see
Figure \ref{front}.

Now, 
the process $\{F_t, t\geq0 \}$ behaves like a continuous time Markov chain
on $\help$, with $F_0=0$. Suppose that this Markov chain tends to
a stationary distribution $\Pi=(\pi_0,\pi_1,\pi_2,\ldots)$ on $\help$. 
Then at a late time $t$ the process $F$ will be in state $0$ with probability 
$\pi_0$. From this state there are two  possible nodes that may be the
next infected ones that result in an increase of the $N$-process, i.e.\ 
the intensity is 2 towards a state that increases the height
of the infection, since these two nodes are associated with two distinct
edges (and thus two distinct exponentially distributed mean 1 random variables). 
With probability $1-\pi_0$ the $F$ process is in some other state
($F\geq 1$) from where there is intensity 1 towards a state that increases the
height of the infection. Thus, knowing the stationary distribution (or rather knowing $\pi_0$)
gives us the 
percolation rate
at a late time as
\begin{equation} \label{heightint}
 2\pi_0+1(1-\pi_0)=1+\pi_0.
\end{equation}

Below we calculate this stationary distribution $\Pi$. To do so, 
it turns out that we may express any $\pi_n$ in terms of $\pi_0$ in the form
$\pi_n=a_n\pi_0-b_n$ where both sequences of (positive) coefficients, $a_n$ and $b_n$, 
satisfy a certain recursion (Claim 1) and
can be expressed in terms of the Bessel functions of the first and second kind. 
We stumbled upon the solution, and
may not have solved it otherwise, when it was noted that the sequence $b_n$ transformed into 
$B_n=(b_n-b_{n-1})/n$ is part of sequence A058797, as listed in \cite{Integer}, 
which is related to the Bessel functions. This relation to Bessel functions was
a surprise and we see no reason to expect it in this context. This method can be 
generalized to percolation on a more general ``ladder-like'' graph, see
\cite{PR}. The sequence $b_n$ is also a subsequence of A056921.

\begin{figure}[htb] \begin{center}
 \includegraphics[scale=0.59]{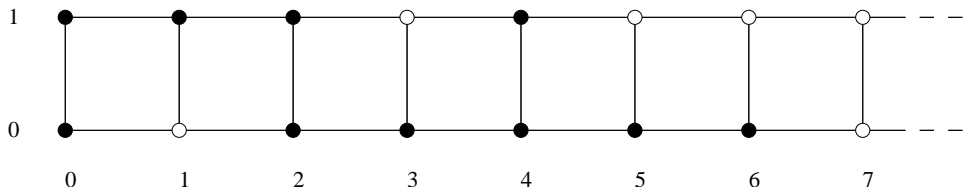}
\end{center}\caption{Infected nodes at time $t$ marked as black. Here
$N_t=6$ and $F_t=2$. }\label{front}\end{figure} 

Now, we start by looking at the intensity matrix $Q$ of the process $F_t$. 
Consider the case in Figure \ref{front}, where $F_t=2$. Now,
there are two edges leading to the node $(5,1)$ and if this is the next infected, the process
will end up in state $F=1$, i.e.\ the intensity
to this state is 2. There is one edge each to node $(6,1)$ and $(7,0)$ and if 
these are the next ones infected,
the process will end up in state $F=0$ or $F=3$, respectively. 
This argument gives us the third row of the intensity matrix $Q$.
We omit the rest of the details, but it is not hard to see that 
\begin{equation} \label{Q}
Q=\left( \begin{array}{rrrrrrr}
-2 & 2 & 0 & 0 & 0  \\
 2 & -3 & 1 & 0 & 0 \\
 1 & 2 & -4 & 1 &0  &\ldots \\
 1 & 1 & 2 & -5 &1 \\
 1 & 1 & 1 & 2 & -6 \\
 1 & 1 & 1 & 1 & 2 &   \\
  &  &\vdots  &&  &  \ddots
\end{array} \right). \end{equation}
$\{F_t,t\geq 0\}$ is irreducible, so if we can find a distribution 
$\Pi=(\pi_0,\pi_1,\pi_2,\ldots)$ on $\help$ 
such that  $\Pi Q=0$
then $\Pi$ is the unique stationary distribution
for $\{F_t, t\geq 0\}$, by  Markov chain theory.
For calculations we henceforth assume that such a distribution $\Pi$ exists.

Using the equations arising from columns one to three of $Q$, the condition $\Pi Q=0$ and
$\sum\pi_j=1$ we get
\begin{equation}\label{first}
 \pi_1=3\pi_0-1,\quad \pi_2 = 11\pi_0-5,\quad\mbox{and}\quad \pi_3 = 56\pi_0-26.
\end{equation}

Also, taking differences of equations arising from adjacent columns of $Q$, from column three
and upwards, gives the equations
\begin{equation}\label{C}
 \quad 0 = \pi_{n-3}-(n+1)\pi_{n-2}+(n+3)\pi_{n-1}-\pi_{n}, \quad n\geq 4.  
\end{equation}

We set $\pi_n=a_n\pi_0-b_n$, where $a_1,a_2,a_3,b_1,b_2,b_3$ are
defined by (\ref{first}) for $n=1,2,3$ and, for $n\geq 4$ by
(\ref{C}) with $a_n$ or $b_n$ in place of $\pi_n$. Table \ref{aochbtabell}
gives the first nine numbers in the sequences $\{a_n\}$ and $\{b_n\}$.	

\begin{table}[hbt]\begin{center}\begin{tabular}{c||c c c c c c c c c } 
 $n$ & 1 & 2 & 3 & 4 & 5 & 6 & 7 & 8 & 9 \\ \hline \hline
 $a_n$ & 3 & 11 & 56 & 340 & 2395 &  19231 &  173490 & 1737706 & 19136803  \\ \hline
 $b_n$ & 1 & 5 & 26 & 158 &  1113  &  8937  & 80624 & 807544 & 8893225
\end{tabular}\end{center}\caption{The beginning of the sequences $\{a_n\}$ and $\{b_n\}$.}
\label{aochbtabell}\end{table}

\begin{claim}
 The sequences $\{a_n\}$ and $\{b_n\}$ both satisfy the recursion
\begin{equation}\label{aob}
 c_n=\frac{c_{n+1}-c_n}{n+1}-\frac{c_{n}-c_{n-1}}{n},\quad n\geq 2.
\end{equation}
\end{claim}
\begin{proof}
We will use induction. First, check initial conditions, $a_2=11$ and $(56-11)/3-(11-3)/2=11$,
$b_2=5$ and $(26-5)/3-(5-1)/2=5$, so equation (\ref{aob}) holds for $n=2$.
Rearranging (\ref{aob}) gives
\begin{equation}\label{an}
 c_{n+1}-c_n=(n+1)\left(c_n+\frac 1n[c_n-c_{n-1}]\right),\quad n\geq 2,
\end{equation}
and this relation is what we will prove.
So assume (\ref{an}) holds for some $n\geq 2$. We aim to show that
it also holds for $n+1$. So, as $n+2\geq 4$, we can use 
(\ref{C}). Adding $\pi_n-\pi_{n-1}$ to both sides of, and substituting $n+2$ for $n$ in 
(\ref{C}) gives
\begin{align}
 c_{n+2}&-c_{n+1}=c_{n-1}-(n+3)c_n+(n+5)c_{n+1}-c_{n+1} \nonumber \\
&=(n+2)\bigg(c_{n+1}+\frac{1}{n+2}\big[2c_{n+1}-(n+2)c_n-\underbrace{(c_n-c_{n-1})}_{(*)}\big]\bigg). \label{it}
\intertext{From (\ref{an}), and the induction hypothesis, we get $c_n-c_{n-1}=(nc_{n+1}-n(n+2)c_n)/(n+1)$, 
which we apply on $(*)$ above, and}
(\ref{it})&= {\textstyle (n+2)\left(c_{n+1}+\frac{1}{n+2}\left[\left(2-\frac{n}{n+1}\right)c_{n+1} 
-\left(n+2-\frac{n(n+2)}{n+1} \right)c_n \right]\right) } \nonumber \\
&=(n+2)\left(c_{n+1}+\frac1{n+1}[c_{n+1}-c_n]\right) . \nonumber
\end{align}
\end{proof}
Now set $C_n=(c_n-c_{n-1})/n$. Then, if $c_n$ 
satisfies (\ref{aob}), i.e.\ $c_n=C_{n+1}-C_n$, this implies
both 
\begin{equation}\label{storaCn}
 C_{n+1}=c_n+C_n  
\end{equation}
and $C_n=C_{n+1}-c_n$. Substitute
$n-1$ for $n$ in the last equation to get 
\begin{equation}\label{storaCn-1}
 C_{n-1}=C_n-c_{n-1}.
\end{equation}
Then adding (\ref{storaCn}) and (\ref{storaCn-1}) and replacing $c_n-c_{n-1}$ with $nC_n$
gives us the following.
\begin{claim}
$A_n=(a_n-a_{n-1})/n$ and $B_n=(b_n-b_{n-1})/n$ both satisfy the
recursion
\begin{equation}\label{rec1}
 C_{n+1}+C_{n-1}=(n+2)C_n,\quad n\geq 3.
\end{equation}
\end{claim}

\subsection{Bessel functions}
The functions
\begin{equation*}
J_n(x)=\sum_{k=0}^\infty \frac{(-1)^k}{k!(n+k)!}\left(\frac x2\right)^{n+2k} 
\end{equation*}
and
\begin{align*}
 Y_n(x)=\frac 1\pi\bigg[2\left(\gamma+\ln\frac x2\right)J_n(x) -
\sum_{k=0}^{n-1}\frac{(n-k-1)!}{k!}\left(\frac x2\right)^{2k-n}  \\ 
-  \sum_{k=0}^\infty(\varsigma_k+\varsigma_{k+n})\frac{(-1)^k}{k!(n+k)!}\left(\frac x2\right)^2\bigg],
\end{align*}
where $n\in\help$, $\varsigma_m=\sum_{j=1}^m\frac 1j$ and $\gamma=0.577\ldots$ is Euler's constant,
are known as Bessel functions of first and second kind. They both satisfy
the recursion
\begin{equation}\label{rec2}
 C_{n+1}(x)+C_{n-1}(x)=\frac{2n}{x}C_n(x).
\end{equation}
Define the function
\begin{equation}\label{Upsdef}
 \Ups(n,m)=\pi[J_{n}(2)Y_m(2)-J_m(2)Y_{n}(2)].
\end{equation}
This function inherits recursion (\ref{rec2}) in
parameter $n$ with $x=2$, because for $m$ fixed 
$\Ups$ is a linear combination of $J_n(2)$ and $Y_n(2)$, so that
\begin{equation} \label{Ups}
\Ups(n+1,m)+\Ups(n-1,m)=n\Ups(n,m). 
\end{equation}

\begin{claim}
 For all integers $n$ and $m$ the function $\Ups(n,m)$ is integer valued.
\end{claim}
\begin{proof}
Clearly $\Ups(m,m)=0$ so it suffices to show that $\Ups(m+1,m)$
is integer valued, since recursion (\ref{Ups}) takes care of all other
values.

Two important relations for the Bessel functions are needed. The first is that
\begin{equation}\label{fir1}
 xJ_n'(x)=nJ_n(x)-xJ_{n+1}(x),\quad \forall n.
\end{equation}
and the same holds with $Y_n$ replacing $J_n$, see
 3.2(4), p.\ 45 and 3.56(4), p.\ 66 of \cite{Wat22}, resp.
The second is a fact relating to the Wronskian, namely that
\begin{equation}\label{sec2}
 \frac2{\pi x}=J_n(x)Y'_n(x)-J_n'(x)Y_n(x), \quad \forall n, x\neq 0,
\end{equation}
see 3.63(1), p.\ 76 of \cite{Wat22}.
By dividing both sides of equation (\ref{fir1}) by $x=2$ and using the resulting
expression  for $J_n'(2)$ and $Y_n'(2)$, resp., in (\ref{sec2}) results in
\begin{align*}
 \frac 1\pi =J_n(2)\left[\frac n2Y_n(2)-Y_{n+1}(2)\right]-\left[\frac n2J_n(2)-J_{n+1}(2)\right]Y_n(2) \\
=J_{n+1}(2)Y_n(2)-J_n(2)Y_{n+1}(2)= \frac1\pi\Ups(n+1,n),
\end{align*}
the last equality by the definition (\ref{Upsdef}) of $\Ups$. Hence, $\Ups(n+1,n)=1$. 
\end{proof}

\begin{table}[hbt]\begin{center}\begin{tabular}{c||c c c c c c c} 
$n$ & 1 & 2 & 3 & 4 & 5 & 6 & 7 \\ \hline \hline
$b_n$ & 1 & 5 & 26 & 158 & 1113 & 8937 & 80624 \\
$B_n=(b_n-b_{n-1})/n$ &  & 2 & 7 & 33 & 191 & 1304 & 10241 \\
$\Ups(n,0)$ & 1 & 1 & 1 & 2 & 7 & 33 & 191 \\ \hline
$a_n$ & 3 & 11 & 56 & 340 & 2395 & 19231 & 173490 \\
$A_n=(a_n-a_{n-1})/n$ &  & 4 & 15 & 71 & 411 & 2806 & 22037\\ 
$2\Ups(n,3)+\Ups(n,0)$ & -3 & -1 & 1 & 4 & 15 & 71 & 411
\end{tabular}\end{center}\caption{Connection between
parameters $a_n, b_n$ and $\Ups(n,m)$.}\label{tabellen} \end{table}

Now we can evaluate the function $\Ups(n,m)$ for different values
of $n$ and $m$ to find that $B_n=\Ups(n+2,0)$, $n=2,3$, and $A_n=2\Ups(n+2,3)+\Ups(n+2,0)$,
$n=2,3$, see Table \ref{tabellen}. Since $\Ups(n+2,m)$ satisfies recursion
(\ref{rec1}) in $n$ this is enough to know that they agree for all $n\geq 2$.

By relation (\ref{storaCn}), definition (\ref{Upsdef}) and the above expression
for $A_n$ and $B_n$ we get
\begin{align}
 b_n &=\Ups(n+3,0)-\Ups(n+2,0)\quad\quad\mbox{and} \nonumber \\
 a_n &=2[\Ups(n+3,3)-\Ups(n+2,3)]+\Ups(n+3,0)-\Ups(n+2,0). \label{anochbn} 
\end{align}

Henceforth we abbreviate $J_n=J_n(2)$ and $Y_n=Y_n(2)$, hence e.g.\
\[ J_n=\sum_{k=0}^\infty \frac{(-1)^k}{k!(k+n)!}. \]
We need  asymptotics for $a_n$ and $b_n$.
\begin{claim}
\begin{equation}\label{asym}
 \lim_{n\to\infty}\frac{b_n}{(n+2)!}=J_0\quad\mbox{and}\quad
 \lim_{n\to\infty}\frac{a_n}{(n+2)!}=2J_3+J_0.
\end{equation}
\end{claim}
\begin{proof}
 \cite{H64} lists the following asymptotic relations, for fixed $x$ and
$n$ tending to infinity,
\[ J_n(x) \sim \frac{1}{\sqrt{2\pi n}}\left(\frac{ex}{2n}\right)^n
 \quad\mbox{and}\quad Y_n(x)\sim -\sqrt\frac{2}{\pi n}\left(\frac{2n}{ex}\right)^n, \]
hence, by the well-known Stirling formula 
$n!\sim \sqrt{2\pi}n^{n+1/2}e^{-n}$ we conclude
that $J_n \sim 1/n!$ and $Y_n \sim -(n-1)!/\pi$. Thus
\[ \Upsilon(n,m) \sim \pi\frac1{n!}Y_m+\pi J_m\frac{(n-1)!}{\pi}\sim J_m\cdot(n-1)!, \]
from which (\ref{asym}) follows.
\end{proof}

Now we are ready to find an expression for
$\pi_0$. Since $\pi_n\to0$ as $n\to\infty$ we get
\begin{align}\label{pi0} 
\pi_0&=
\lim_{n	\to\infty}\frac{\pi_n+b_n}{a_n}=\frac{J_0}{2J_3+J_0}=0.4647184275\ldots 
\end{align}

This in turn gives us 
\begin{align}\label{pin}
 \pi_n &= a_n\pi_0-b_n \nonumber \\
&=2[\Upsilon(n+3,3)-\Upsilon(n+2,3)]\pi_0+[\Upsilon(n+3,0)-\Upsilon(n+2,0)][\pi_0-1] \nonumber \\
&=\frac{2\pi}{2J_3+J_0}
\bigg([J_{n+3}Y_{3}-J_{3}Y_{n+3}-J_{n+2}Y_{3}+J_{3}Y_{n+2}]J_0 \nonumber \\
&\phantom{=\frac{2\pi}{(J_3Y_0-J_0Y_3)}\bigg(}
-[J_{n+3}Y_{0}-J_{0}Y_{n+3}-J_{n+2}Y_{0}+J_{0}Y_{n+2}]J_3 \bigg) \nonumber \\
&=\frac{2\pi (J_3Y_0-J_0Y_3)}{2J_3+J_0}(J_{n+2}-J_{n+3}) 
=\frac{2\Ups(3,0)}{2J_3+J_0}(J_{n+2}-J_{n+3}) \nonumber \\
&=\frac{2}{2J_3+J_0}(J_{n+2}-J_{n+3}), \quad n\geq 1.
\end{align}
Then $\sum_0^N \pi_j=1-2J_{N+3}/(2J_3+J_0)$ and as $J_n\to 0$, as $n\to\infty$,
we can verify that $\Pi$ is indeed a distribution.

\subsection{The time constant}
Following the discussion before and after (\ref{heightint}) we get
the time constant
\begin{equation} \label{timeconstant}
\tau=\frac1{1+\pi_0}=\frac{2J_3+J_0}{2J_3+2J_0}=0.6827250759\ldots 
\end{equation}

This can be compared with first-passage percolation on $\help$, i.e.\
the infinite path graph (with
exponential r.v.'s with mean 1), where the time 
constant necessarily is 1. 
We may also compare it with 
$1/2$ which is the time constant if the rungs of the ladder
where associated with random variables all identically $0$,
i.e.\ there is immediate infection between nodes on different levels at distance
$1$ from each other. 

Note that we started the process by infecting both $(0,0)$ and $(0,1)$.
Consider the infection of the single
node $(0,0)$ at time $0$. 
Then, the front process $\hat F$ will  not be a continuous time Markov chain (as before), as the
behaviour of the front is very different before and after the random time
$Z$ until there are infected nodes
on both levels. Since we are interested in stationarity we may simply disregard 
what happens before $Z$ and view this as starting the front process $F$ at time $Z$,
$F_t=\hat F_{Z+t}$, $t\geq 0$. 
Then $F$ is a continuous time Markov chain with intensity matrix given by (\ref{Q}),
as before, albeit started in a random state instead of $F_0=0$. 
$Z$ is dominated by a mean 1 exponential random variable
(the  edge between $(0,0)$ and $(0,1)$). As $F$ will converge to the
stationary distribution $\Pi$, so will $\hat F$. 

Hence, infecting just a single node initially makes no difference to the time constant.
(In fact, starting with an initial infection of any finite number of nodes
will not change the time constant.) 

This time constant also gives an upper bound for the time constant for
first-passage percolation with mean 1 exponential r.v.'s on $\mathbb Z^2$,
but not stronger than already existing bounds, see \cite{AP02}.

\subsection{Residual times}
The fact that we can write down the stationary distribution $\Pi$ for 
the process $F_t$ also gives us the opportunity to calculate another statistic
for first passage percolation on the ladder, namely the average residual time. 
As far as we know, this quantity has not previously been studied for percolation. 
Define the \emph{residual time at $t$} as $R_t=\inf\{s: N_{t+s}=N_t+1\}$, i.e.\ the
time it takes, after $t$, for the infection to spread one more step up the ladder.
We are interested in $T=\ex R_t$ for a late time $t$. 

One might a priori think that $T$ and $\tau$ are the same thing. If we were
looking at percolation on $\help$, which is the same as a Poisson process, then they are. This is
the waiting time paradox; on one hand: looking at a fixed time point increases the
probability of choosing a long time interval,
but on the other: typically half that interval has passed. For a Poisson process, these effects balance
out. For percolation on a ladder, as shown  in (\ref{timeconstant}) and (\ref{av.resid.time}) below,
they do not. 

Now, to calculate $T$ we will assume that $F$ has its stationary distribution, so that
we may disregard time and put $R=R_t$. 
Let $\gamma_n =\ex[R | F=n]$. Hence, $\gamma_n$ is the
expected time it takes for the infection to spread another step given that
the front is in state $n$.

If $F=0$ then the residual time is the minimum of two mean 1 exponential random variables,
so that $\gamma_0=1/2$. For $n\geq 1$, we write down a recursive formula for $\gamma_n$.
If $F=n$ then there are $n+2$ possible edges, equally likely, 
that the infection might spread along.
Only one of these results in an increase in the height of the infection, namely the one that
gives $F=n+1$. Two edges result in $F=n-1$ and the remaining ones 
to one of the states $\{0,1,\ldots,n-2\}$, respectively.   
This gives the formula
\begin{equation}\label{gamman}
 \gamma_n=\frac{1}{n+2}\left(1+2\gamma_{n-1}+\sum_{j=0}^{n-2}\gamma_j\right).
\end{equation}
By writing down the formula for $\gamma_{n-1}$ one sees that
\[1+\sum_0^{n-2}\gamma_j=(n+1)\gamma_{n-1}-\gamma_{n-2}\] which inserted
into (\ref{gamman}) yields the incremental relation
\[ (n+2)(\gamma_{n}-\gamma_{n-1})=\gamma_{n-1}-\gamma_{n-2}, \]
which together with the fact that $\gamma_1=\frac13(1+2\frac12)=\frac23$ in turn gives us the formulas
\begin{equation}\label{incogamma}
 \gamma_n-\gamma_{n-1}=\frac 1{(n+2)!},\;n\geq 1,\quad\mbox{and}\quad \gamma_n=\sum_{j=0}^{n+2}\frac{1}{j!},\;n\geq 0.
\end{equation}

Then, as $F_t$ is in state $j$ with probability $\pi_j$ where it remains on average $\gamma_j$
amount of time before $N_t$ increases, we get the average residual time $T$ as
\begin{align} 
T &=\sum_{n=0}^\infty \pi_n\gamma_n
    =\frac{1}{2J_3+J_0}\left(J_0/2+2\sum_{n=1}^\infty (J_{n+2}-J_{n+3})\gamma_n\right) \nonumber \\
&=\frac{1}{2J_3+J_0}\left( \frac 12J_0+ 2J_3\gamma_1 + 2\sum_{n=1}^\infty(\gamma_{n+1}-\gamma_{n})J_{n+3} \right) \nonumber \\
&=\frac{1}{2J_3+J_0}\left(\frac12J_0+\frac 43J_3+2\sum_{n=1}^\infty \frac{J_{n+3}}{(n+3)!}\right) 
=0.5953444665\ldots \label{av.resid.time}
\end{align}

\vspace{0.3cm}
\noindent\textbf{Acknowledgements:} Thanks to Sven Erick Alm
 for introducing me to this problem and to Svante Janson for making  useful suggestions
regarding the mathematics. Also, we wish to thank an anonymous referee
for carefully reading the manuscript and providing us with numerous helpful comments.

\end{document}